\documentclass[11pt,oneside, english]{article}
\usepackage{dsfont}
\usepackage{amsmath, amssymb, amsthm, verbatim,enumerate, bbm}
\usepackage{enumerate}
\usepackage[utf8x]{inputenc}
\usepackage[T1]{fontenc}
\usepackage[english]{babel}
\usepackage{mathtools}
\usepackage[
letterpaper,
top=0.8in,
bottom=0.8in,
left=1in,
right=1in]{geometry}
\usepackage{bm} %
\usepackage[colorlinks=true, allcolors=blue]{hyperref}

\usepackage{graphicx}

\graphicspath{{./figures/}}
\usepackage{epsfig}
\usepackage{fullpage}
\usepackage{mathrsfs}
\usepackage{enumerate}
\usepackage{verbatim}
\usepackage{mhequ}
\usepackage{algorithm}
\usepackage{algorithmicx}
\usepackage{algpseudocode}
\usepackage[nameinlink,capitalise,noabbrev]{cleveref}

\usepackage{showlabels}
\numberwithin{equation}{section}

\def \be{\begin{equs}}
	\def \ee{\end{equs}}

\def \be{\begin{equs}}
\def \ee{\end{equs}}

\newtheorem{theorem}{Theorem}[section]

\newtheorem*{namedtheorem}{\theoremname}
\newcommand{\theoremname}{testing}

\newtheorem{lemma}[theorem]{Lemma}

\newtheorem{proposition}[theorem]{Proposition}

\newtheorem{corollary}[theorem]{Corollary}

\newtheorem*{question*}{Question}

\theoremstyle{definition}
\newtheorem{definition}[theorem]{Definition}

\theoremstyle{plain}

\title{A PHASE TRANSITION FOR REPEATED K-AVERAGES}

\author{
    Rohit Chaudhuri\thanks{Indian Statistical Institute, Kolkata}
    }

\date{}

\begin{document}

\pagestyle{empty}

\maketitle
\begin{abstract}
    Let $x_1,\dots,x_{n}$ be a fixed sequence of real numbers. At each stage, pick $k$ integers $\{I_{i}\}_{1\leq i \leq k}$ uniformly at random without replacement and then for each $i \in \{1,2,\dots,k\}$ replace $x_{I_i}$ by $(x_{I_1}+x_{I_2}+\dots+x_{I_k})/k$. It is easy to observe that all the co-ordinates converge to $(x_1+\dots+x_n)/n$. In this article, we extend the result of \cite{chatterjee2019note} by establishing the mixing time to be in between $\frac{n}{k \log k}\log n$ and $\frac{n}{k-1}\log n$.
\end{abstract}
\section{Main Result}
In this article, we will mostly focus on the initial condition $x_{0}=(1,0,\dots,0)$. We denote the $L^{1}$ distance from convergence by $T(k)=\sum_{i=1}^{n}|x_{k,i}-\bar{x}_{0}|.$ The next theorem formally describes the ``mixing behaviour" of the $L^{1}$ metric.
\begin{theorem}\label{thm:1}
For $x_{0}=(1,0,\dots,0),$ in probability as $n \to \infty$ we have $$T(\theta n \log n) \to 2$$ for any $\theta < \frac{1}{k \log k}$, and $$T(\theta n \log n) \to 0$$ for any $\theta > \frac{1}{k-1}$.
\end{theorem}
\section{PROOFS}
Define a Markov chain $\{x_m\}^{\infty}_{m=0}$ as in the previous section. Without loss of generality, we safely can assume that $\bar{x}_{0}=0.$ We will establish our results under this assumption.
\subsection{Decay of expected $L^{2}$ distance}
For each $l$, let $$S_{n}(l):=\sum_{i=1}^{n}x^{2}_{l,i}. $$Let $\mathcal{F}_{l}$ be the $\sigma$-algebra by the history up to time $l$.
\begin{proposition}\label{prop:1}
For any $l \geq 0$,
\be
\mathbb{E}\left(S_{n}(l+1)|\mathcal{F}_{l}\right)=\left(1-\frac{k-1}{n-1}\right)S_{n}(l).
\ee
\end{proposition}
\begin{proof}
For any $i \in [n]$, we perform the following computations
\be
\mathbb{E}\left[x^{2}_{l+1,i}|\mathcal{F}_{l}\right]&=\frac{{n-1\choose k}}{{n\choose k}}x^{2}_{l,i}+\frac{1}{{n\choose k}}\sum_{1 \leq i_1 < i_2\dots <i_{k-1}\leq n}\left(\frac{x_{l,i}+x_{l,i_1}+\dots+x_{l,i_{k-1}}}{k}\right)^2,\\
&=\left(\frac{n-k}{n}\right)x^{2}_{l,i}+\frac{{n-1\choose k-1}}{k^2{n\choose k}}x^{2}_{l,i}+\frac{{n-2\choose k-2}}{k^2{n\choose k}}\sum^{n}_{j=1 : j \neq i}x^{2}_{l,j}+\frac{2{n-2\choose k-2}}{k^2{n\choose k}}\sum^{n}_{j=1: j \neq i}x_{l,j}x_{l,i}+\frac{2{n-3\choose k-3}}{k^2{n\choose k}}\sum_{1\leq p<q \leq n:\text{ } p,q\neq n}x_{l,p}x_{l,q},\\
\ee
Since $\sum^{n}_{i=1}x_{l,i}=0,$ we have 
\be
\mathbb{E}\left[x^{2}_{l+1,i}|\mathcal{F}_{l}\right]&=\left(1-\frac{k}{n}+\frac{1}{nk}\right)x^{2}_{l,i}+\frac{k(k-1)\left(S_{n}(l)-x^{2}_{l,i}\right)}{k^2n(n-1)}-\frac{2k(k-1)x^{2}_{l,i}}{k^2n(n-1)}+\frac{2{n-3\choose k-3}}{k^2{n\choose k}}\left(2x^{2}_{l,i}-S_{n}(l)\right),\\
&=\left(1-\frac{k}{n}+\frac{1}{nk}-\frac{3(k-1)}{kn(n-1)}+\frac{2(k-1)(k-2)}{nk(n-1)(n-2)}\right)x^{2}_{l,i}+\left(\frac{(k-1)}{kn(n-1)}-\frac{(k-1)(k-2)}{kn(n-1)(n-2)}\right)S_{n}(l),\\
\ee
Summing over $i$, we obtain
\be
\mathbb{E}\left[S_{n}(l+1)|\mathcal{F}_{l}\right]&=\left(1-\frac{k}{n}+\frac{1}{nk}-\frac{3(k-1)}{kn(n-1)}+\frac{2(k-1)(k-2)}{nk(n-1)(n-2)}\right)S_{n}(l)+\left(\frac{(k-1)}{k(n-1)}-\frac{(k-1)(k-2)}{k(n-1)(n-2)}\right)S_{n}(l),\\
&=\left(1-\frac{k-1}{n-1}\right)S_{n}(l).
\ee
This provides us with the required identity.
\end{proof}
\begin{corollary}\label{coro:1}
If we denote $\tau:=1-\frac{k-1}{n-1},$ we have $\mathbb{E}\left(S_{n}(l)\right)=\tau^{l}S_{n}(0).$ Moreover, $\lim S_{n}(l)/\tau^{l}$ exists and is finite almost surely.
\end{corollary}
\begin{proof}
The first claim follows from \cref{prop:1} and induction. Now note that $M_{l}:=S_{n}(l)/\tau^{l}$ is a non-negative martingale, and so its limit exists and is finite almost surely.
\end{proof}
\phantomsection
  \addcontentsline{toc}{section}{References}
  \bibliographystyle{amsalpha}
  \bibliography{biblio.bib}
\end{document}